\documentclass{daj}

\usepackage{amsmath, amsfonts, amsthm}
\usepackage[utf8]{inputenc}
\usepackage{graphicx}

\newtheorem{lemma}{Lemma}[section]
\newtheorem{theorem}[lemma]{Theorem}

\newcommand{\Q}{\mathbb{Q}}
\newcommand{\R}{\mathbb{R}}

\newcommand{\Fcal}{\mathcal{F}}

\newcommand{\Pcal}{\mathcal{P}}
\newcommand{\Rcal}{\mathcal{R}}

\DeclareMathOperator{\conv}{conv}         
\DeclareMathOperator{\cutp}{CUT_\square}  

\newcommand{\rey}{R}                   
\newcommand{\orto}{\mathrm{O}}         
\newcommand{\alphaGW}{\alpha_{\mathrm{GW}}}  
\newcommand{\fGW}{f_{\mathrm{GW}}}
\newcommand{\tGW}{t_{\mathrm{GW}}}
\newcommand{\GW}{{\mathrm{GW}}}

\DeclareMathOperator{\sdp}{SDP}        

\newcommand{\optprob}[1]{{\arraycolsep=0pt%
  \begin{array}{r@{\ }l@{\quad}l}
    #1
  \end{array}}}

\newcommand{\defi}[1]{\textit{#1}}

\dajAUTHORdetails{%
  title = {On the Integrality Gap of the Maximum-Cut Semidefinite Programming Relaxation in Fixed Dimension},
  author = {Fernando M\'ario de Oliveira Filho and Frank Vallentin},
  plaintextauthor = {Fernando Mario de Oliveira Filho and Frank Vallentin},
  plaintexttitle = {On the Integrality Gap of the Maximum-Cut
    Semidefinite Programming Relaxation in Fixed Dimension},  
  runningtitle = {The Integrality Gap of the Maximum-cut Semidefinite
    Relaxation in Fixed Dimension},
  runningauthor = {F.M. de Oliveira Filho and F. Vallentin},
  keywords = {maximum-cut problem, semidefinite programming, integrality gap},
}   

\dajEDITORdetails{%
   year={2020},
   number={10},
   received={9 August 2018},   
   published={6 August 2020},  
   doi={10.19086/da.14164},       
}   

\begin{document}

\begin{frontmatter}[classification=text]


\author[fmof]{Fernando M\'ario de Oliveira Filho}
\author[frank]{Frank Vallentin\thanks{This project has received
  funding from the European Union's Horizon 2020 research and
  innovation programme under the Marie Sk\l{}odowska-Curie agreement
  No 764759. The second author is partially supported by the SFB/TRR 191
  ``Symplectic Structures in Geometry, Algebra and Dynamics'' and by
  the project ``Spectral bounds in extremal discrete geometry''
  (project number 414898050), both funded by the DFG.}}

\begin{abstract}
  We describe a factor-revealing convex optimization problem for the
  integrality gap of the maximum-cut semidefinite programming
  relaxation: for each $n \geq 2$ we present a convex optimization
  problem whose optimal value is the largest possible ratio between
  the value of an optimal rank-$n$ solution to the relaxation and the
  value of an optimal cut. This problem is then used to compute lower
  bounds for the integrality gap.
\end{abstract}
\end{frontmatter}


\section{Introduction}

For~$x$, $y \in \R^n$, write $x \cdot y = x_1 y_1 + \cdots + x_n y_n$
for the Euclidean inner product.
Let~$S^{n-1} = \{\, x \in \R^n : x \cdot x = 1\,\}$ be the
$(n-1)$-dimensional unit sphere. Given a nonempty finite set~$V$, a
nonnegative matrix~$A \in \R^{V \times V}$, and an integer~$n \geq 1$,
write
\begin{equation}
  \label{eq:sdp-def}
  \sdp_n(A) = \max\biggl\{\, \sum_{x, y \in V} A(x, y) (1 - f(x) \cdot f(y))
  : f\colon V \to S^{n-1}\,\biggr\}.
\end{equation}
Replacing~$S^{n-1}$ above by~$S^\infty$, the set of all
sequences~$(a_k)$ such that~$\sum_{k=0}^\infty a_k^2 = 1$, we obtain
the definition of~$\sdp_\infty(A)$.

Given a finite (loopless) graph~$G = (V, E)$ and a nonnegative weight
function~$w\colon E \to \R_+$ on the edges of~$G$, the
\defi{maximum-cut problem} asks for a set~$S \subseteq V$ that
maximizes the weight
\[
  \sum_{e \in \delta(S)} w(e) = \sum_{\substack{x \in S, y \in V
      \setminus S\\xy \in E}} w(xy)
\]
of the cut $\delta(S) = \{\, e \in E : |e \cap S| =
1\,\}$. If~$A\colon V \times V \to \R$ is the matrix such
that~$A(x, y) = w(xy)$ when~$xy \in E$ and~$A(x, y) = 0$ otherwise,
then the weight of a maximum cut is~$(1/4) \sdp_1(A)$.

$\sdp_n(A)$ is actually the optimal value of a
semidefinite program with a rank constraint, namely
\begin{equation}
  \label{eq:sdp-problem}
  \optprob{\text{max}&\sum_{x, y \in V} A(x, y) (1 - M(x, y))\\
    &M(x, x) = 1\quad\text{for~$x \in V$},\\
    &\text{$M \in \R^{V \times V}$ is positive semidefinite and has
      rank at most~$n$.}
  }
\end{equation}
In~$\sdp_\infty(A)$ the rank constraint is simply dropped. The
optimization problem $\sdp_\infty(A)$ is the
\defi{semidefinite programming relaxation} of the maximum-cut problem.

Obviously,~$\sdp_\infty(A) \geq \sdp_1(A)$. In a fundamental paper,
Goemans and Williamson~\cite{GoemansW1995} showed that, if~$A$ is a
nonnegative matrix, then
\[
  \sdp_1(A) \geq \alphaGW \sdp_\infty(A),
\]
where
\[
  \alphaGW = \min_{t \in [-1, 1]} \frac{1 - (2 / \pi) \arcsin t}{1 -
    t} = 0.87856\ldots.
\]

The \defi{$n$-dimensional integrality gap} of the
semidefinite programming relaxation is
\[
  \gamma_n = \sup \Bigl\{\, \frac{\sdp_n(A)}{\sdp_1(A)} :
  \text{$A$ is a nonnegative matrix}\,\Bigr\},
\]
but it is often more natural to work with its
reciprocal~$\alpha_n = \gamma_n^{-1}$.  Goemans and Williamson thus
showed that~$\alpha_\infty \geq \alphaGW$; Feige and
Schechtman~\cite{FeigeS2001} later showed
that~$\alpha_\infty = \alphaGW$ (see also~\S8.3 in G\"artner and
Matou\v{s}ek~\cite{GartnerM2012}).

In dimension~2 it is known that
\begin{equation}
  \label{eq:alpha-2}
  \alpha_2 = \frac{32}{25 + 5\sqrt{5}} = 0.88445\ldots.
\end{equation}
The `$\leq$' direction was shown by Delorme and
Poljak~\cite{DelormeP1993b, DelormeP1993a}; the `$\geq$' direction was
shown by Goemans in an unpublished note (cf.~Avidor and
Zwick~\cite{AvidorZ2005}, who also provide another proof of this
result).  Avidor and Zwick~\cite{AvidorZ2005} showed
that~$\alpha_3 \geq 0.8818$. Except for~$n = 2$ and~3, it is an open
problem whether~$\alpha_n > \alpha_\infty = \alphaGW$.


\subsection{A factor-revealing optimization problem}
\label{sec:factor-revealing}

Theorem~\ref{thm:main} below gives a factor-revealing optimization
problem for~$\alpha_n$: an optimization problem defined for
each~$n \geq 2$ whose optimal value is~$\alpha_n$. Relaxations of it
can be solved by computer to give upper bounds for~$\alpha_n$, as done
in~\S\ref{sec:bounds}.

For a finite and nonempty set~$U$, write
\[
  \cutp(U) = \conv\{\, f \otimes f^* : f\colon U \to \{-1,1\}\,\},
\]
where~$f\otimes f^*$ is the external product of the vector~$f$, that
is, the matrix whose entry~$(x, y)$ is~$f(x) f(y)$. This set is known
as the \defi{cut polytope} and was extensively
investigated~\cite{DezaL1997}.

A \defi{kernel} is a square-integrable (with respect to the Lebesgue
measure) real-valued function on~$S^{n-1} \times S^{n-1}$; the set of
continuous kernels is denoted by $C(S^{n-1} \times S^{n-1})$. Write
\begin{multline*}
  \cutp(S^{n-1}) = \{\, K \in C(S^{n-1} \times S^{n-1}) : \bigl(K(x,
  y)\bigr)_{x,y \in U} \in \cutp(U)\\
  \text{for every finite and nonempty~$U \subseteq S^{n-1}$}\,\}.
\end{multline*}
In principle, it is not clear that anything other than the constant~1
kernel is in~$\cutp(S^{n-1})$. If~$f\colon S^{n-1} \to \{-1,1\}$ is
any nonconstant measurable function, then one could be tempted to say
that~$(x, y) \mapsto f(x) f(y)$ is in~$\cutp(S^{n-1})$, but no such
kernel is continuous, so to see that~$\cutp(S^{n-1})$ is nontrivial
requires a more complicated argument. Fix~$e \in S^{n-1}$ and
let~$f(x)$ be~$1$ if~$e \cdot x \geq 0$ and~$-1$
otherwise; let~$K(x, y) = (2/\pi) \arcsin x\cdot y$ for every~$x$, $y
\in S^{n-1}$. Grothendieck's identity says that
\[
  K(x, y) = \int_{\orto(n)} f(Tx) f(Ty)\, d\mu(T),
\]
where~$\orto(n)$ is the group of~$n \times n$ orthogonal matrices
and~$\mu$ is the Haar measure on~$\orto(n)$ normalized so the total
measure is~1. Then it is easy to see that~$K$ is continuous and that
it belongs to~$\cutp(S^{n-1})$. This kernel was first considered by
McMillan~\cite{McMillan1955}, who was perhaps the first to use such an
infinite-dimensional analogue of the cut polytope.

We say that a kernel~$K$ is \defi{invariant} if~$K(Tx, Ty) = K(x, y)$
for all~$T \in \orto(n)$ and~$x$, $y \in S^{n-1}$. An invariant kernel
is in fact a univariate function, since the value of~$K(x, y)$ depends
only on the inner product~$x \cdot y$. Hence for~$t \in [-1, 1]$ we
write~$K(t)$ for the common value taken by~$K$ on pairs~$(x, y)$ with
inner product~$t$.

\begin{theorem}
  \label{thm:main}
  If~$n \geq 2$, then~$\alpha_n$ is the optimal value of
  \begin{equation}
    \label{eq:main-prob}
    \optprob{\sup&\alpha\\
      &1 - K(t) \geq \alpha(1 - t)\quad\text{for all~$t \in [-1,
        1]$},\\
      &\text{$K \in \cutp(S^{n-1})$ is invariant}.
    }
  \end{equation}
\end{theorem}

This theorem is similar to the integral representation for the
Grothendieck constant~\cite[Theorem~3.4]{Pisier2012}. The easy
direction is to show that the optimal value of~\eqref{eq:main-prob} is
at most~$\alpha_n$.

\begin{proof}[Proof of the easy direction of Theorem~\ref{thm:main}]
Let $(K, \alpha)$ be a feasible solution of~\eqref{eq:main-prob} and
let $A \in \R^{V \times V}$ be any nonnegative matrix and
$f\colon V \to S^{n-1}$ be a function achieving the maximum in
$\sdp_n(A)$. Note
$\bigl(K(f(x), f(y))\bigr)_{x, y \in V} \in \cutp(V)$. This implies
that there are nonnegative numbers~$\lambda_1$, \dots,~$\lambda_r$
that sum up to~1 and functions~$f_1$,
\dots,~$f_r\colon V \to \{-1,1\}$ such that
\[
  K(f(x), f(y)) = K(f(x) \cdot f(y)) = \lambda_1 f_1(x) f_1(y) +
  \cdots + \lambda_r f_r(x) f_r(y)
\]
for all~$x$, $y \in V$. But then
\[
  \begin{split}
    \sdp_1(A) &\geq \sum_{k=1}^r \lambda_k \sum_{x, y \in V} A(x, y) (1 -
    f_k(x) f_k(y))\\
    &= \sum_{x, y \in V} A(x, y) (1 - K(f(x) \cdot f(y)))\\
    &\geq \alpha \sum_{x, y \in V} A(x, y) (1 - f(x) \cdot f(x))\\
    &=\alpha\sdp_n(A),
  \end{split}
\]
so~$\alpha \leq \alpha_n$.
\end{proof}

A proof that the optimal value of~\eqref{eq:main-prob} is at
least~$\alpha_n$ is given in~\S\ref{sec:proof}, but it works only
for~$n \geq 3$. For~$n = 2$ a direct proof can be given by showing a
feasible solution of~\eqref{eq:main-prob} with objective
value~$\alpha_2$; this was done, in a different language, by Avidor
and Zwick~\cite{AvidorZ2005} and is outlined in~\S\ref{sec:alpha-2},
where a short discussion on how lower bounds for~$\alpha_n$ can be
found is also presented.

Notice that the optimization problem~\eqref{eq:main-prob} is infinite:
the kernel~$K$ lies in an infinite-dimensional space and must satisfy
infinitely many constraints, not to mention that the separation
problem over~$\cutp(U)$ is NP-hard since the maximum-cut problem is
NP-hard~\cite{Karp1972}. In~\S\ref{sec:bounds} we will see how~$K$ can
be parameterized and how the problem can be relaxed (by relaxing the
constraint that~$K$ must be in~$\cutp(S^{n-1})$) and effectively
discretized so it can be solved by computer, providing us with
upper bounds for~$\alpha_n$. From feasible solutions of this
relaxation, instances with large integrality gap can be constructed,
as shown in~\S\ref{sec:bad-instances}.


\section{Proof of Theorem~\ref{thm:main} for~$n \geq 3$}
\label{sec:proof}

The difficult part of the proof is to show that the optimal value
of~\eqref{eq:main-prob} is at least~$\alpha_n$. This is done here
for~$n \geq 3$, and for this we need a few lemmas.

Let~$\mu$ be the Haar measure on the orthogonal group~$\orto(n)$,
normalized so the total measure is~1. The \defi{Reynolds
  operator}~$\rey$ projects a kernel~$K$ onto the space of invariant
kernels by averaging:
\[
  \rey(K)(x, y) = \int_{\orto(n)} K(Tx, Ty)\, d\mu(T)
\]
for all~$x$, $y \in S^{n-1}$. If~$K$ is a continuous kernel, then so
is~$\rey(K)$~\cite[Lemma~5.4]{DeCorteOV2018}, and
if~$f\in L^2(S^{n-1})$, then~$\rey(f \otimes f^*)$ is
continuous~\cite[Lemma~5.5]{DeCorteOV2018}, where~$f \otimes f^*$ is
the kernel mapping~$(x, y)$ to~$f(x) f(y)$.

A function~$f\colon S^{n-1} \to \R$ \defi{respects} a
partition~$\Pcal$ of~$S^{n-1}$ if~$f$ is constant on
each~$X \in \Pcal$; we write~$f(X)$ for the common value of~$f$
in~$X$.

\begin{lemma}
  \label{lem:additive-error}
  If~$n \geq 2$, then for every~$\eta > 0$ there is a
  partition~$\Pcal$ of~$S^{n-1}$ into finitely many measurable sets
  such that for every finite set~$I \subseteq [-1, 1]$ and every
  nonnegative function~$z\colon I \to \R$ there is a
  function~$f\colon S^{n-1} \to \{-1, 1\}$ that respects~$\Pcal$ and
  satisfies
  \begin{equation}
    \label{eq:partition-ineq}
    \sum_{t \in I} z(t) (1 - R(f \otimes f^*)(t)) \geq \sum_{t
      \in I} z(t)(\alpha_n(1 - t) - \eta).
  \end{equation}    
\end{lemma}

\begin{proof}
Let~$\Pcal$ be any partition of~$S^{n-1}$ into finitely many
measurable sets of small enough diameter such that for all~$X$,
$Y \in \Pcal$, $x$, $x' \in X$, and~$y$, $y' \in Y$, we have
$|x \cdot y - x' \cdot y'| \leq \alpha_n^{-1}\eta$. Such a partition
can be obtained by considering e.g.~the Voronoi cell of each point of
an $\epsilon$-net for~$S^{n-1}$ for small enough~$\epsilon$.

For~$u \in S^{n-1}$ and~$X \in \Pcal$, write
\[
  [u, X] = \{\, T \in \orto(n) : Tu \in X\,\}.
\]
Then~$[u, X]$ is measurable~\cite[Theorem~3.7]{Mattila1995}, so
$\{\, [u, X] : X \in \Pcal\,\}$ is a partition of~$\orto(n)$ into
measurable sets, and hence for any~$u$, $v \in S^{n-1}$ so is the
common refinement
\[
  \{\, [u, X] \cap [v, Y] : \text{$(X, Y) \in \Pcal \times \Pcal$ and
    $[u, X] \cap [v, Y] \neq \emptyset$}\,\}.
\]

Write~$u = (1, 0, \ldots, 0) \in S^{n-1}$ and for~$t \in [-1, 1]$ let
$v_t = (t, (1 - t^2)^{1/2}, 0, \ldots, 0)$, so~$u \cdot v_t =
t$. If~$f\colon S^{n-1} \to \R$ respects~$\Pcal$, then for every
finite~$I \subseteq [-1, 1]$ and every nonnegative~$z\colon I \to \R$
we have
\[
  \begin{split}
    \sum_{t \in I} z(t) (1 - R(f \otimes f^*)(t)) &=
    \sum_{t \in I} z(t) \int_{\orto(n)} 1 - f(Tu) f(Tv_t)\, d\mu(T)\\
    &=\sum_{t \in I} z(t) \sum_{X, Y \in \Pcal} \int_{[u, X] \cap
      [v_t, Y]} 1 - f(Tu) f(Tv_t)\, d\mu(T)\\
    &=\sum_{t \in I} z(t) \sum_{X, Y \in \Pcal} (1 - f(X) f(Y))
    \mu([u, X] \cap [v_t, Y])\\
    &=\sum_{X, Y \in \Pcal} (1 - f(X) f(Y)) \sum_{t \in I} z(t)
    \mu([u, X] \cap [v_t, Y]).
  \end{split}
\]

By considering the matrix~$A_z\colon \Pcal \times \Pcal \to \R$ such
that
\begin{equation}
  \label{eq:A-matrix}
  A_z(X, Y) = \sum_{t \in I} z(t) \mu([u, X] \cap [v_t, Y]),
\end{equation}
we see that finding a function~$f\colon S^{n-1} \to \{-1,1\}$ that
respects~$\Pcal$ and maximizes the left-hand side
of~\eqref{eq:partition-ineq} is the same as finding an optimal
solution of~$\sdp_1(A_z)$, so there is such a function~$f$
satisfying
\begin{equation}
  \label{eq:good-f}
  \sum_{t \in I} z(t) (1 - R(f \otimes f^*)(t)) = \sdp_1(A_z).
\end{equation}

Now let~$g\colon\Pcal \to S^{n-1}$ be such that~$g(X) = x$ for
some~$x \in X$ chosen arbitrarily. Recall that the sets in~$\Pcal$
have small diameter, so
\[
  \begin{split}
    \sdp_n(A_z) &\geq \sum_{X, Y \in \Pcal} A_z(X, Y) (1 - g(X) \cdot
    g(Y))\\
    &=\sum_{t \in I} z(t) \sum_{X, Y \in \Pcal} (1 - g(X) \cdot g(Y))
    \mu([u, X] \cap [v_t, Y])\\
    &=\sum_{t \in I} z(t) \sum_{X, Y \in \Pcal} \int_{[u, X] \cap
      [v_t, Y]} 1 - g(X) \cdot g(Y)\, d\mu(T)\\
    &\geq \sum_{t \in I} z(t) \sum_{X, Y \in \Pcal} \int_{[u, X] \cap
      [v_t, Y]} 1 - (Tu) \cdot (Tv_t) - \alpha_n^{-1}\eta\, d\mu(T)\\
    &=\sum_{t \in I} z(t) \int_{\orto(n)} 1 - t -
    \alpha_n^{-1}\eta\, d\mu(T)\\
    &=\sum_{t \in I} z(t)((1 - t) -\alpha_n^{-1} \eta).
  \end{split}
\]

Now take any finite~$I \subseteq [-1, 1]$ and any
nonnegative~$z\colon I \to \R$. If~$f$ is a function that
respects~$\Pcal$ and for which~\eqref{eq:good-f} holds, then
\[
    \sum_{t \in I} z(t) (1 - R(f \otimes f^*)(t)) = \sdp_1(A_z)
    \geq \alpha_n \sdp_n(A_z)
    \geq \sum_{t \in I} z(t)(\alpha_n(1-t) - \eta),
\]
as we wanted.  
\end{proof}

Lemma~\ref{lem:additive-error} is enough to show the following weaker
version of the difficult direction of Theorem~\ref{thm:main}:

\begin{lemma}
  \label{lem:weaker-main}
  If~$n \geq 2$ and~$1 \geq \delta > 0$, then the optimal value of the
  optimization problem
  \begin{equation}
    \label{eq:main-prob-weaker}
    \optprob{\sup&\alpha\\
      &1 - K(t) \geq \alpha(1 - t)\quad\text{for all~$t \in [-1,
        1-\delta]$},\\
      &\text{$K \in \cutp(S^{n-1})$ is invariant}
    }
  \end{equation}
  is at least~$\alpha_n$.
\end{lemma}

\begin{proof}
Fix~$\eta > 0$ and let~$\Pcal$ be a partition supplied by
Lemma~\ref{lem:additive-error}. Let~$\Fcal$ be the set of all
functions~$f\colon S^{n-1} \to \{-1,1\}$ that respect~$\Pcal$;
note~$\Fcal$ is finite.

Let~$I_1 \subseteq I_2 \subseteq \cdots$ be a sequence of finite
nonempty subsets of~$[-1, 1]$ whose union is the set of all rational
numbers in~$[-1, 1]$. Suppose there is no~$m_k\colon \Fcal \to \R$
satisfying
\[
  \begin{split}
    &\sum_{f \in \Fcal} (1 - R(f \otimes f^*)(t)) m_k(f) \geq 
    \alpha_n(1-t) - \eta\qquad\text{for all~$t \in I_k$},\\
    &\sum_{f \in \Fcal} m_k(f) = 1,\\
    &m_k \geq 0.
  \end{split}
\]
Farkas's lemma~\cite[\S7.3]{Schrijver1986} says that, if this system
has no solution, then there is~$z\colon I_k \to \R$, $z \geq 0$,
and~$\rho \in \R$ such that
\[
  \begin{split}
    &\rho + \sum_{t \in I_k} z(t) (1 - R(f \otimes f^*)(t)) \leq
    0\qquad\text{for all~$f \in \Fcal$},\\
    &\rho + \sum_{t \in I_k} z(t) (\alpha_n (1-t) - \eta) > 0.
  \end{split}
\]
Together, these inequalities imply that for every~$f \in \Fcal$ we
have
\[
  \sum_{t \in I_k} z(t)(1 - R(f \otimes f^*)(t)) < \sum_{t \in I_k}
  z(t)(\alpha_n(1 - t) - \eta),
\]
a contradiction to the choice of~$\Pcal$.

Since all~$m_k$ lie in~$[0,1]^\Fcal$, which is a compact set, the
sequence~$(m_k)$ has a converging subsequence; say this subsequence
converges to~$m\colon\Fcal \to \R$. Then~$m \geq 0$ and
$\sum_{f \in \Fcal} m(f) = 1$. Moreover,
\begin{equation}
  \label{eq:m-ineq}
  \sum_{f \in \Fcal} (1 - R(f \otimes f^*)(t)) m(f) \geq 
  \alpha_n(1-t) - \eta\qquad\text{for all~$t \in [-1, 1]$}.
\end{equation}
Indeed, the inequality holds for all~$t \in [-1, 1] \cap \Q$.
But~$R(f \otimes f^*)$ is continuous for every~$f$, so the left-hand
side above is a continuous function of~$t$, whence the inequality
holds for every~$t \in [-1, 1]$.

Fix~$1 \geq \delta > 0$ and~$\epsilon > 0$ and
set~$\eta = \alpha_n \epsilon \delta$; let~$m$ be such
that~\eqref{eq:m-ineq} holds. If~$t \leq 1-\delta$,
then~$1-t\geq \delta$ and
\[
  (1-\epsilon)(1-t) = (1-t)-\epsilon(1-t)\leq (1-t)-\epsilon\delta.
\]
So, for~$t \in [-1, 1-\delta]$, the left-hand side of~\eqref{eq:m-ineq}
is at least
\[
  \alpha_n(1-t)-\eta = \alpha_n ((1-t)-\alpha_n^{-1}\eta) \geq
  \alpha_n (1 - \epsilon) (1 - t).
\]

Now $K_\epsilon = \sum_{f \in \Fcal} R(f \otimes f^*) m(f)$ is a
continuous kernel that moreover belongs to~$\cutp(S^{n-1})$. So for
every~$\epsilon > 0$ there is $K_\epsilon \in \cutp(S^{n-1})$ such
that $(K_\epsilon, \alpha_n (1-\epsilon))$ is a feasible solution
of~\eqref{eq:main-prob-weaker}, and by letting~$\epsilon$ approach~0
we are done.
\end{proof}

For~$n \geq 3$, Theorem~\ref{thm:main} can be obtained from
Lemma~\ref{lem:weaker-main} by using the following lemma.

\begin{lemma}
  \label{lem:upper-interval}
  For every~$n \geq 3$, there is~$1 \geq \delta > 0$ such that
  if~$(K, \alpha)$ is any feasible solution
  of~\eqref{eq:main-prob-weaker}, then
  \[
    1 - K(t) \geq \alpha(1-t)\qquad\text{for all~$t \in [1-\delta,
      1]$.}
  \]
\end{lemma}

The proof of this lemma uses some properties of Jacobi polynomials,
and goes through only for~$n \geq 3$. A proof of
Theorem~\ref{thm:main} for~$n = 2$ is given in~\S\ref{sec:alpha-2}.

The Jacobi polynomials\footnote{See for example the book by
  Szeg\H{o}~\cite{Szego1975} for background on orthogonal polynomials.}
with parameters~$(\alpha, \beta)$, $\alpha$, $\beta > -1$, are the
orthogonal polynomials with respect to the weight function
$(1-t)^\alpha (1+t)^\beta$ on the interval~$[-1,1]$. We denote the
Jacobi polynomial with parameters~$(\alpha, \beta)$ and degree~$k$
by~$P_k^{(\alpha, \beta)}$ and normalize it
so~$P_k^{(\alpha, \beta)}(1) = 1$.

A continuous kernel~$K\colon S^{n-1} \times S^{n-1} \to \R$ is
\defi{positive} if $\bigl(K(x, y)\bigr)_{x,y \in U}$ is positive
semidefinite for every finite and nonempty set~$U \subseteq S^{n-1}$.
Schoenberg~\cite{Schoenberg1942} characterizes continuous, positive,
and invariant kernels via their expansions in terms of Jacobi
polynomials:

\begin{theorem}[Schoenberg's theorem]
  \label{thm:schoenberg}
  A kernel~$K\colon S^{n-1} \times S^{n-1} \to \R$ is continuous,
  positive, and invariant if and only if there are
  numbers~$a_k \geq 0$ satisfying~$\sum_{k=0}^\infty a_k < \infty$
  such that
  \[
    K(x, y) = \sum_{k=0}^\infty a_k P_k^{(\nu, \nu)}(x \cdot
    y)\qquad\text{for all~$x$, $y \in S^{n-1}$}
  \]
with absolute and uniform convergence, where~$\nu = (n-3)/2$.
\end{theorem}

Schoenberg's theorem is used in the proof of
Lemma~\ref{lem:upper-interval} and again in~\S\S\ref{sec:alpha-2}
and~\ref{sec:bounds}.

\begin{proof}[Proof of Lemma~\ref{lem:upper-interval}]
Fix~$n \geq 3$ and set~$\nu = (n - 3) / 2$. Claim: there
is~$1 \geq \delta > 0$ such that~$t = P_1^{(\nu, \nu)}(t) \geq
P_k^{(\nu, \nu)}(t)$ for all~$k \geq 2$ and~$t \in
[1-\delta,1]$.

The lemma quickly follows from this claim. Indeed, say~$(K, \alpha)$
is feasible for~\eqref{eq:main-prob-weaker}. Since every matrix
in~$\cutp(U)$ for finite~$U$ is positive semidefinite, every kernel
in~$\cutp(S^{n-1})$ is positive. Hence using Schoenberg's theorem we
write
\[
  K(t) = \sum_{k=0}^\infty a_k P_k^{(\nu,
    \nu)}(t)\qquad\text{for all~$t \in [-1, 1]$}.
\]
Since~$K \in \cutp(S^{n-1})$, we have~$K(1) = 1$,
so~$\sum_{k=0}^\infty a_k = 1$.

As~$(K, \alpha)$ is a feasible solution
of~\eqref{eq:main-prob-weaker}, we know that
\[
  1 - K(-1) \geq \alpha (1 - (-1)) = 2\alpha.
\]
Now~$|P_k^{(\nu, \nu)}(t)| \leq 1$ for all~$k$ and all~$t \in [-1,
1]$, so~$K(-1) \geq a_0 - (1 - a_0)$, whence~$a_0 \leq 1-
\alpha$. The claim implies that, if~$t \in [1-\delta, 1]$, then
\[
  K(t) \leq a_0 + (1 - a_0)t,
\]
so for~$t \in [1-\delta, 1]$ we have
\[
  1 - K(t) \geq 1 - a_0 - (1 - a_0)t = 1 - t - a_0(1-t) \geq 1 - t -
  (1-\alpha)(1-t) = \alpha(1-t),
\]
as we wanted.

To prove the claim, we use the following integral representation of
Feldheim and Vilenkin for the Jacobi polynomials: for~$\nu \geq 0$,
\begin{multline}
  \label{eq:int-rep}
  P_k^{(\nu, \nu)}(\cos\theta) =
  \frac{2\Gamma(\nu+1)}{\Gamma(1/2)\Gamma(\nu + 1/2)} \int_0^{\pi/2}
  \cos^{2\nu}\phi (1 - \sin^2\theta\cos^2\phi)^{k/2}\\
  \cdot
  P_k^{(-1/2,-1/2)}(\cos\theta(1-\sin^2\theta\cos^2\phi)^{-1/2})\,
  d\phi.
\end{multline}
This formula is adapted to our normalization of the Jacobi polynomials
from Corollary~6.7.3 in the book by Andrews, Askey, and
Roy~\cite{AndrewsAR1999}; see also equation~(3.23) in the thesis by
Oliveira~\cite{Oliveira2009}.

For fixed~$\theta$ and~$\phi$, the function
$k \mapsto (1 - \sin^2\theta\cos^2\phi)^{k/2}$ is monotonically
decreasing. Write~$t = \cos\theta$ and recall that the Jacobi
polynomials are bounded by~1 in~$[-1, 1]$; plug~$k = 2$ in the
right-hand side of~\eqref{eq:int-rep} to get
\begin{equation}
  \label{eq:int-upper-bound}
  P_k^{(\nu, \nu)}(t) \leq \frac{2\Gamma(\nu+1)}{\Gamma(1/2)\Gamma(\nu
    + 1/2)} \int_0^{\pi/2} \cos^{2\nu}\phi (1 -
  (1-t^2)\cos^2\phi)\, d\phi
\end{equation}
for all~$t \in [0,1]$ and~$k \geq 2$. For~$\nu = (n-3)/2$
with~$n \geq 4$, we show that there is~$\delta > 0$ such that the
right-hand side above is at most~$t$ for all~$t \in [1-\delta, 1]$;
the case~$n = 3$ will be dealt with shortly.

Let~$m \geq 2$ be an integer. Write~$\cos^m\phi = \cos^{m-1}\phi
\cos\phi$ and use integration by parts to get
\[
  m\int_0^{\pi/2} \cos^m\phi\, d\phi = (m-1) \int_0^{\pi/2} \cos^{m-2}
  \phi\, d\phi.
\]
It follows by induction on~$m$ that, if~$\nu = (n-3)/2$
with~$n \geq 3$, then
\begin{equation}
  \label{eq:int-value}
  \int_0^{\pi/2} \cos^{2\nu}\phi\, d\phi = \frac{\Gamma(1/2) \Gamma(\nu +
    1/2)}{2\Gamma(\nu+1)}.
\end{equation}

The right-hand side of~\eqref{eq:int-upper-bound} is a degree-2
polynomial on~$t$; let us denote it
by~$p_\nu$. Use~\eqref{eq:int-value} to get
\[
  p_\nu(t) = \frac{2\nu + 1}{2(\nu + 1)} t^2 + \frac{1}{2(\nu+1)}.
\]
It is then a simple matter to check that, for~$\nu = (n - 3) / 2$
with~$n \geq 4$, there is~$\delta > 0$ such that~$p_\nu(t) \leq t$ for
all~$t \in [1-\delta, 1]$.

For~$n = 3$ and hence~$\nu = 0$, we have~$p_\nu(t) \geq t$ for
all~$t \in [0, 1]$. In this case, we may take~$k = 4$
in~\eqref{eq:int-rep} and follow the same reasoning, proving that the
degree~4 polynomial obtained will have the desired property. It then
only remains to show that~$P_2^{(0, 0)}$ and~$P_3^{(0, 0)}$ are
below~$P_1^{(0, 0)}$ for~$t$ close enough to~1, and this can be done
directly.
\end{proof}

All that is left to do is to put it all together.

\begin{proof}[Proof of Theorem~\ref{thm:main} for~$n \geq 3$]
In~\S\ref{sec:factor-revealing} we have seen that the optimal value
of~\eqref{eq:main-prob} is at most~$\alpha_n$. The reverse inequality
follows from Lemmas~\ref{lem:weaker-main} and~\ref{lem:upper-interval}
put together.
\end{proof}


\section{Lower bounds for~$\alpha_n$ and a proof of
  Theorem~\ref{thm:main} for~$n = 2$}
\label{sec:alpha-2}

To get a lower bound for~$\alpha_n$, one needs to show a feasible
solution of~\eqref{eq:main-prob}. One such feasible solution, that
shows that~$\alpha_n \geq \alphaGW$, is~$(K_\GW, \alphaGW)$ with
\begin{equation}
  \label{eq:groth}
  K_\GW(x\cdot y) = (2/\pi) \arcsin x\cdot y.
\end{equation}
We encountered this kernel in the introduction. Fix~$e \in S^{n-1}$
and let~$f_\GW\colon S^{n-1} \to \{-1, 1\}$ be such that~$\fGW(x) = 1$
if~$e\cdot x \geq 0$ and~$-1$ otherwise. Recall that Grothendieck's
identity is
\[
  K_\GW(x\cdot y) = R(\fGW \otimes \fGW^*)(x \cdot y),
\]
whence in particular~$K_\GW \in \cutp(S^{n-1})$.

Let~$\tGW \in [-1, 1]$ be such that
$\alphaGW = (1 - K_\GW(\tGW)) / (1 - \tGW)$;
then~$\tGW = -0.68918\ldots$. The easy direction of the following
result is implicit in the work of Avidor and Zwick~\cite{AvidorZ2005}.

\begin{theorem}
  \label{thm:lower-bound}
  If~$n \geq 2$, then~$\alpha_n > \alphaGW$ if and only if there is an
  invariant kernel~$K \in \cutp(S^{n-1})$ such that
  \begin{equation}
    \label{eq:part-ineq}
    1 - K(\tGW) > 1 - K_\GW(\tGW).
  \end{equation}
  If, moreover,~$\alpha_n > \alphaGW$, then there is a measurable
  function~$f\colon S^{n-1} \to \{-1,1\}$ such
  that~\eqref{eq:part-ineq} holds for~$K = R(f \otimes f^*)$.
\end{theorem}

\begin{proof}
First the easy direction. Suppose there is such a kernel~$K$. Then
\begin{equation}
  \label{eq:f-better}
  1 - K(\tGW) > 1 - K_\GW(\tGW) = \alphaGW (1 - \tGW).
\end{equation}
Both functions
\[
  t \mapsto 1 - K(t)\qquad\text{and}\qquad
  t \mapsto 1 - K_\GW(t)
\]
are continuous in~$[-1,1]$. From~\eqref{eq:f-better}, we see that
there is~$\epsilon > 0$ such that the first function above is at
least~$(\alphaGW+\epsilon) (1 - t)$ in some interval~$I$
around~$\tGW$. The second function above is at least~$\alphaGW(1 - t)$
in~$[-1,1]$ and, if~$\epsilon$ is small enough, then it is at
least~$(\alphaGW + \epsilon)(1-t)$ in~$[-1,1] \setminus I$ (recall
from~\eqref{eq:groth} that we know the second function
explicitly). But then for some~$\lambda \in [0, 1]$ and small
enough~$\epsilon' > 0$ we will have that
\[
  K' = \lambda K + (1 - \lambda) K_\GW \in \cutp(S^{n-1})
\]
is such that~$1 - K'(t) \geq (\alphaGW+\epsilon') (1 - t)$ for
all~$t \in [-1, 1]$, so the optimal value of~\eqref{eq:main-prob}
is greater than~$\alphaGW$ and therefore~$\alpha_n > \alphaGW$ from
the easy direction of Theorem~\ref{thm:main} (proved
in~\S\ref{sec:factor-revealing}).

Now suppose~$\alpha_n > \alphaGW$. For every~$\eta > 0$,
Lemma~\ref{lem:additive-error} gives a measurable
function~$f\colon S^{n-1} \to \{-1,1\}$ such that
\[
  1 - R(f \otimes f^*)(\tGW) \geq \alpha_n (1 - \tGW) - \eta
\]
(take $I = \{\tGW\}$ and~$z = 1$ in the lemma);
set~$K = R(f \otimes f^*)$. Then
\[
    1 - K(\tGW) \geq \alpha_n (\alphaGW^{-1} (1 -
    K_\GW(\tGW))) - \eta.
\]
Since~$\alpha_n / \alphaGW > 1$, we finish by taking~$\eta$ close
enough to~0.
\end{proof}

Theorem~\ref{thm:lower-bound} shows that, to find a lower bound
for~$\alpha_n$, we need to find a better partition of the
sphere~$S^{n-1}$, and this can be done by finding a maximum cut in a
graph defined on a discretization of the sphere (cf.~the proof of
Lemma~\ref{lem:additive-error}). This can be tricky in general: Avidor
and Zwick~\cite{AvidorZ2005} present such a better partition
for~$n = 3$, but their construction is \textit{ad hoc}. For~$n = 2$,
however, one may use the hyperplane rounding procedure to obtain such
a better partition, in a curious application of the Goemans-Williamson
algorithm to improve on itself.

We want to find an invariant kernel~$K \in \cutp(S^{n-1})$
satisfying~\eqref{eq:f-better}, that is, we want to find a good
solution of the following optimization problem:
\[
  \optprob{\sup&1 - K(\tGW)\\
    &\text{$K \in \cutp(S^{n-1})$ is invariant.}
  }
\]
This seems to be a difficult problem, but we can relax the constraint
that~$K \in \cutp(S^{n-1})$ by requiring only that~$K$ be
positive. Then, using Schoenberg's theorem to parameterize~$K$ as
in~\S\ref{sec:proof}, we get the following relaxation of our problem:
\begin{equation}
  \label{eq:dim-2-relax}
  \optprob{\sup&1 - \sum_{k=0}^\infty a_k P_k^{(\nu, \nu)}(\tGW)\\[3pt]
    &\sum_{k=0}^\infty a_k = 1,\\
    &a_k \geq 0\quad\text{for all~$k \geq 0$.}
  }
\end{equation}

For~$n = 2$ and hence~$\nu = -1/2$, the optimal solution
of~\eqref{eq:dim-2-relax} is~$a_k = 0$ for all~$k \neq 4$
and~$a_4 = 1$, as may be proved, for instance, by showing a solution
to the dual of~\eqref{eq:dim-2-relax} having the same objective value
as the solution~$a$ (see~\S\ref{sec:bounds} for a description of the
dual problem of a problem related to~\eqref{eq:dim-2-relax}).

Using formula~(5.1.1) from Andrews, Askey, and
Roy~\cite{AndrewsAR1999}, this means that the optimal kernel is
\[
  K(\cos\theta) = P_4^{(-1/2, -1/2)}(\cos\theta) = \cos4\theta.
\]
If we identify the circle~$S^1$ with the interval~$[0, 2\pi]$, then
the inner product between points~$\theta$, $\phi \in [0, 2\pi]$
is~$\arccos (\theta - \phi)$, so 
\[
  K(\theta, \phi) = \cos 4(\theta-\phi) = \cos 4\theta\cos4\phi + \sin
  4\theta \sin 4\phi.
\]
Taking~$g\colon S^1 \to S^1$ such that
$g(\theta) = (\cos 4\theta, \sin 4\theta)$, we have
$K(\theta, \phi) = g(\theta) \cdot g(\phi)$.

Now, let us round the rank-2 solution~$g$. Let~$e = (1, 0)$ and
set~$f(\theta) = 1$ if~$e \cdot g(\theta) \geq 0$ and~$f(\theta) = -1$
otherwise. The resulting partition is exactly the windmill partition
that, combined with the partition~$\fGW$ of the sphere into two equal
halves, shows that
\[
  \alpha_2 = \frac{32}{25 + 5 \sqrt{5}}
\]
(cf.~Avidor and Zwick~\cite{AvidorZ2005}); see also
Figure~\ref{fig:windmill}.

\begin{figure}[tb]
\begin{center}
\includegraphics{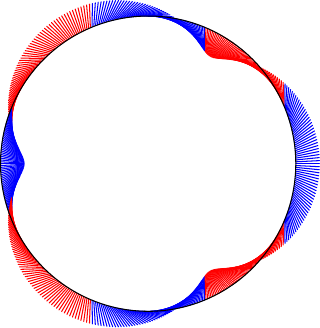}\hskip1cm
\includegraphics{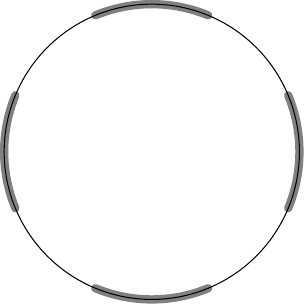}
\end{center}

\caption{On the left we have the unit circle~$S^1$ and, supported on a
  point~$\theta$, the
  vector~$f(\theta) = (\cos 4\theta, \sin 4\theta)$ --- in blue
  if~$\cos4\theta \geq 0$ and in red otherwise. On the right, we have
  in gray the segments of the circle where~$\cos4\theta \geq 0$;
  this is the windmill partition.}
\label{fig:windmill}
\end{figure}

\begin{proof}[Proof of Theorem~\ref{thm:main} for~$n = 2$]
In~\S\ref{sec:factor-revealing} we have seen that the optimal value
of~\eqref{eq:main-prob} is at most~$\alpha_2$. The reverse inequality
is proved by Avidor and Zwick~\cite{AvidorZ2005}: they show how to
pick~$\lambda \in [0, 1]$ such that, if~$f\colon S^1 \to \{-1,1\}$ is
the windmill partition of Figure~\ref{fig:windmill}
and~$\fGW\colon S^1 \to \{-1,1\}$ is the partition into two equal
halves, then~$(K, \alpha)$ with
\[
  K = \lambda R(f \otimes f^*) + (1 - \lambda) R(\fGW\otimes \fGW^*)
\]
and
\[
  \alpha = \frac{32}{25 + 5\sqrt{5}}
\]
is a feasible solution of~\eqref{eq:main-prob}.
Since~$\alpha_2 = \alpha$, we are then done.
\end{proof}

For~$n \geq 3$, the approach outlined above does not work. The optimal
solution of the relaxation~\eqref{eq:dim-2-relax} is always~$a_k = 0$
for all~$k \neq 1$ and~$a_1 = 1$. The hyperplane rounding then gives
the partition~$\fGW$ into two equal halves, therefore not providing a
lower bound for~$\alpha_n$ better than~$\alphaGW$.


\section{Upper bounds for~$\alpha_n$ and bad instances}
\label{sec:bounds}

Let us see how to solve a relaxation of~\eqref{eq:main-prob} in order
to get upper bounds for~$\alpha_n$. The first order of business is to
use Schoenberg's theorem (Theorem~\ref{thm:schoenberg}) to
parameterize~$K$ as
\begin{equation}
  \label{eq:K-expr}
  K(t) = \sum_{k=0}^\infty a_k P_k^{(\nu,
    \nu)}(t)\qquad\text{for all~$t \in [-1, 1]$},
\end{equation}
where~$\nu = (n - 3) / 2$, $a_k \geq 0$ for all~$k$,
and~$\sum_{k=0}^\infty a_k < \infty$.

Say now that~$U \subseteq S^{n-1}$ is a nonempty finite set
and~$Z\colon U \times U \to \R$ and~$\beta \in \R$ are such that
\[
  \sum_{x, y \in U} Z(x, y) A(x, y) \geq \beta
\]
for all~$A \in \cutp(U)$, so~$Z$ and~$\beta$ give a valid constraint
for~$\cutp(U)$. If~$K \in \cutp(S^{n-1})$, then
\[
  \sum_{x, y \in U} Z(x, y) K(x, y) \geq \beta.
\]
Rewriting this inequality using the parametrization of~$K$ we see that
the variables~$a_k$ satisfy the constraint
\[
  \sum_{k=0}^\infty a_k r_k \geq \beta,
\]
where~$r = (r_k)$ is the sequence such that
\[
  r_k = \sum_{x, y \in U} Z(x, y) P_k^{(\nu, \nu)}(x \cdot y).
\]

Let~$\Rcal$ be a finite collection of pairs~$(r, \beta)$, each one
associated with a valid constraint of~$\cutp(U)$ for some finite
set~$U \subseteq S^{n-1}$, as described above. Recall that,
if~$K \in \cutp(S^{n-1})$, then~$K(1) = 1$, and
that~$P_k^{(\nu, \nu)}(1) = 1$ in our normalization. Choose a finite
nonempty set~$I \subseteq [-1, 1]$. Then the following
linear program with infinitely many variables but finitely
many constraints is a relaxation of~\eqref{eq:main-prob}; its optimal
value thus provides an upper bound for~$\alpha_n$:
\begin{equation}
  \label{eq:relax-primal}
  \optprob{\sup&\alpha\\
    &\sum_{k=0}^\infty a_k = 1,\\
    &\alpha(1-t) + \sum_{k=0}^\infty a_k P_k^{(\nu, \nu)}(t) \leq
    1&\text{for all~$t \in I$},\\
    &\sum_{k=0}^\infty a_k r_k \geq \beta&\text{for all~$(r, \beta) \in
      \Rcal$},\\
    &a_k \geq 0&\text{for all~$k \geq 0$}.
  }
\end{equation}

A dual problem for~\eqref{eq:relax-primal} is
\begin{equation}
  \label{eq:relax-dual}
  \optprob{\inf&\lambda+\sum_{t \in I} z(t) - \sum_{(r, \beta) \in
      \Rcal} y(r, \beta) \beta\\[3pt]
    &\sum_{t \in I} z(t) (1 - t) = 1,\\[3pt]
    &\lambda + \sum_{t \in I} z(t) P_k^{(\nu, \nu)}(t) - \sum_{(r,
      \beta) \in \Rcal} y(r, \beta) r_k \geq 0\quad\text{for all~$k
      \geq 0$},\\[3pt]
    &z, y \geq 0.
  }
\end{equation}
It is routine to show that weak duality holds between the two
problems: if~$(a, \alpha)$ is a feasible solution
of~\eqref{eq:relax-primal} and~$(\lambda, z, y)$ is a feasible
solution of~\eqref{eq:relax-dual}, then
\[
  \alpha \leq \lambda+\sum_{t \in I} z(t) - \sum_{(r, \beta) \in
    \Rcal} y(r, \beta) \beta.
\]
So to find an upper bound for~$\alpha_n$ it suffices to find a
feasible solution of~\eqref{eq:relax-dual}.

To find such a feasible dual solution we follow the same approach
presented by DeCorte, Oliveira, and
Vallentin~\cite[\S8]{DeCorteOV2018} for a very similar problem. We
start by choosing a large enough value~$d$ (say~$d = 2000$) and
truncating the series in~\eqref{eq:K-expr} at degree~$d$,
setting~$a_k = 0$ for all~$k > d$. Then, for finite sets~$I$
and~$\Rcal$, problem~\eqref{eq:relax-primal} becomes a finite
linear program. We solve it and from its dual we obtain a
candidate solution~$(\lambda, z, y)$ for the original,
infinite-dimensional dual. All that is left to do is check that this
is indeed a feasible solution, or else that it can be turned into a
feasible solution by slightly increasing~$\lambda$. This verification
procedure is also detailed by DeCorte, Oliveira, and Vallentin
(ibid.,~\S8.3).

Finding a good set~$I \subseteq [-1, 1]$ is easy: one simply takes a
finely spaced sample of points. Finding a good set~$\Rcal$ of
constraints is another issue. The approach is, again, detailed by
DeCorte, Oliveira, and Vallentin (ibid.,~\S8.3); here is an
outline. We start by setting~$\Rcal = \emptyset$. Then, having a
solution of~\eqref{eq:relax-primal}, and having access to a list of
facets of~$\cutp(U)$ for a set~$U$ of~7 elements, numerical methods
for unconstrained optimization are used to find points on the sphere
for which a given inequality is violated. These violated inequalities
are then added to~\eqref{eq:relax-primal} and the process is repeated.

Table~\ref{tab:bounds} shows a list of upper bounds for~$\alpha_n$
found with the procedure described above. These bounds have been
rigorously verified using the approach of DeCorte, Oliveira, and
Vallentin.

\begin{table}[tb]
  \begin{center}
    \begin{tabular}{cccc}
      $n$&{\sl Upper bound}&$n$&{\sl Upper bound}\\[3pt]
      4  & 0.881693 & 12 & 0.878923 \\
      5  & 0.880247 & 13 & 0.878893 \\
      6  & 0.879526 & 14 & 0.878864 \\
      7  & 0.879184 & 15 & 0.878835 \\
      8  & 0.879079 & 16 & 0.878798 \\
      9  & 0.879016 & 17 & 0.878772 \\
      10 & 0.878981 & 18 & 0.878772 \\
      11 & 0.878953 & 19 & 0.878744
    \end{tabular}
  \end{center}
  \caption{Upper bounds for~$\alpha_n$ from a relaxation of
    problem~\eqref{eq:main-prob}. For~$n = 3$, the relaxation gives an
    upper bound of~$0.8854$, not better than~$\alpha_2$. These bounds
    have all been computed considering a same set~$\Rcal$ with~28
    constraints from the cut polytope found heuristically for the
    case~$n = 4$; improvements could possibly be obtained by trying to
    find better constraints for each dimension. The bound
    using~$\Rcal$ decreases more and more slowly after~$n = 19$;
    for~$n = 10000$ one obtains the upper bound~$0.878695$.}
  \label{tab:bounds}
\end{table}


\subsection{Constructing bad instances}
\label{sec:bad-instances}

A feasible solution of~\eqref{eq:relax-dual} gives an upper bound
for~$\alpha_n$, but this upper bound is not constructive, that is, we
do not get an instance of the maximum-cut problem with large
integrality gap. Let us see now how to extract bad instances for the
maximum-cut problem from a solution of~\eqref{eq:relax-dual}.

Let~$I \subseteq [-1,1)$ be a finite nonempty set of inner products
and~$\Rcal$ be a finite set of constraints from the cut polytope.
Say~$(\lambda, z, y)$ is a feasible solution of~\eqref{eq:relax-dual}
and let
\[
  \alpha = \lambda+\sum_{t \in I} z(t) - \sum_{(r, \beta) \in
    \Rcal} y(r, \beta) \beta
\]
be its objective value.

The intuition behind the construction is simple. We consider a graph
on the sphere~$S^{n-1}$, where~$x$, $y \in S^{n-1}$ are adjacent
if~$x\cdot y \in I$ and the weight of an edge between~$x$ and~$y$
is~$z(x \cdot y)$. Bad instances will arise from discretizations of
this infinite graph.

Given a partition~$\Pcal$ of~$S^{n-1}$ into finitely many sets, denote
by~$\delta(\Pcal)$ the maximum diameter of any set in~$\Pcal$.
Let~$(\Pcal_m)$ be a sequence of partitions of~$S^{n-1}$ into finitely
many measurable sets such that~$\Pcal_{m+1}$ is a refinement
of~$\Pcal_m$ and
\[
  \lim_{m\to\infty} \delta(\Pcal_m) = 0.
\]

For~$m \geq 0$, let~$A_z^m\colon \Pcal_m \times \Pcal_m \to \R$ be the
matrix defined in~\eqref{eq:A-matrix} for the
partition~$\Pcal = \Pcal_m$ and the function~$z$. Since~$\Pcal_{m+1}$
is a refinement of~$\Pcal_m$, both limits
\[
  \lim_{m\to\infty} \sdp_1(A^m_z)\qquad\text{and}\qquad
  \lim_{m\to\infty} \sdp_n(A^m_z)
\]
exist, as the sequences of optimal values are monotonically increasing
and bounded. As~$z \neq 0$, both limits are positive, hence
\begin{equation}
  \label{eq:part-limit}
  \lim_{m\to\infty} \frac{\sdp_1(A^m_z)}{\sdp_n(A^m_z)}
\end{equation}
exists. Claim: the limit above is at most~$\alpha$.

Once the claim is established, we are done: for every~$\epsilon > 0$,
by taking~$m$ large enough (that is, by taking a fine enough
partition) we have
\[
  \frac{\sdp_1(A^m_z)}{\sdp_n(A^m_z)} \leq \alpha + \epsilon,
\]
that is, we get a sequence of bad instances for the maximum-cut problem.

To prove the claim, suppose~\eqref{eq:part-limit} is at
least~$\alpha + \epsilon$ for some fixed~$\epsilon > 0$. Then for all
large enough~$m$ we have
\[
  \sdp_1(A^m_z) \geq (\alpha + \epsilon) \sdp_n(A^m_z).
\]
Following the proof of Lemma~\ref{lem:additive-error}, this means that
for every large enough~$m$ there is a function~$f_m\colon S^{n-1} \to
\{-1,1\}$ that respects~$\Pcal_m$ and satisfies
\[
  \sum_{t \in I} z(t)(1 - R(f_m\otimes f_m^*)(t)) \geq
  \sum_{t \in I} z(t)((\alpha + \epsilon) (1-t) - \eta_m),
\]
where~$\eta_m \geq 0$ and~$\eta_m \to 0$ as~$m \to \infty$.

Use the feasibility of~$(\lambda, z, y)$ for~\eqref{eq:relax-dual}
together with the definition of~$\alpha$ to get from the above
inequality that
\begin{equation}
  \label{eq:almost-contradiction}
  \lambda + \sum_{t \in I} z(t) R(f_m \otimes f_m^*)(t) - \sum_{(r,
    \beta) \in \Rcal} y(r, \beta) \beta \leq -\epsilon + \eta_m\sum_{t \in I} z(t).
\end{equation}

Next, note that~$R(f_m \otimes f_m^*) \in \cutp(S^{n-1})$. Using
Schoenberg's theorem (Theorem~\ref{thm:schoenberg}), write
\[
  R(f_m \otimes f_m^*)(t) = \sum_{k=0}^\infty a_k P_k^{(\nu, \nu)}(t),
\]
where~$\nu = (n - 3) / 2$,~$a_k \geq 0$,
and~$\sum_{k=0}^\infty a_k = 1$. Use again the feasibility
of~$(\lambda, z, y)$ for~\eqref{eq:relax-dual} together
with~\eqref{eq:almost-contradiction} to get
\[
  \begin{split}
    0 &\leq \sum_{k=0}^\infty a_k \biggl( \lambda + \sum_{t \in I} z(t)
    P_k^{(\nu, \nu)}(t) - \sum_{(r, \beta) \in \Rcal} y(r, \beta)
    r_k\biggr)\\
    &\leq \lambda + \sum_{t \in I} z(t) R(f_m \otimes f_m^*)(t) -
    \sum_{(r, \beta) \in \Rcal} y(r, \beta) \beta\\
    &\leq -\epsilon + \eta_m\sum_{t \in I} z(t).
  \end{split}
\]
Since~$\epsilon > 0$ and~$\eta_m \to 0$ as~$m \to \infty$, by
taking~$m$ large enough we get a contradiction, proving the claim.


\section*{Acknowledgements}

We thank the referees for valuable suggestions that improved the
paper.


\bibliographystyle{amsplain}


\begin{dajauthors}
  \begin{authorinfo}[fmof]
    F.M. de Oliveira Filho\\
    Delft Institute of Applied Mathematics\\
    Delft University of Technology\\
    Van Mourik Broekmanweg~6, 2628 XE Delft, The Netherlands.\\
    \url{fmario@gmail.com}
\end{authorinfo}

\begin{authorinfo}[frank]
  F.~Vallentin\\
  Mathematisches Institut\\
  Universit\"at zu K\"oln\\
  Weyertal~86--90, 50931 K\"oln, Germany.\\
  \url{frank.vallentin@uni-koeln.de}
\end{authorinfo}
\end{dajauthors}

\end{document}